\documentclass[11pt]{article}
\usepackage{amssymb,amsmath,mathrsfs,amsthm,indentfirst} % font used for R in Real numbers
\usepackage[numbers,sort&compress]{natbib}
\allowdisplaybreaks

\catcode`!=11
\let\!int\int \def\int{\displaystyle\!int}
\let\!lim\lim \def\lim{\displaystyle\!lim}
\let\!sum\sum \def\sum{\displaystyle\!sum}
\let\!sup\sup \def\sup{\displaystyle\!sup}
\let\!inf\inf \def\inf{\displaystyle\!inf}
\let\!cap\cap \def\cap{\displaystyle\!cap}
\let\!max\max \def\max{\displaystyle\!max}
\let\!min\min \def\min{\displaystyle\!min}
\let\!frac\frac \def\frac{\displaystyle\!frac}
\catcode`!=12

\newtheorem{theorem}{\bf Theorem}[section]
\newtheorem{lemma}{\bf Lemma}[section]

\newtheorem{proposition}{\bf Proposition}[section]
\newtheorem{remark}{\bf Remark}[section]

\def\pd#1#2{\frac{\partial#1}{\partial#2}}

\catcode`@=11 \@addtoreset{equation}{section} 
\catcode`@=12

\begin{document}

\title{
 The Impact of Time Delay and Angiogenesis in a Tumor Model}

\author{Zejia Wang, Haihua Zhou, Huijuan Song\thanks
{Corresponding author. E-mail: {\it songhj@jxnu.edu.cn}.} \quad
\\
{\normalsize School of Mathematics and Statistics,}
\\
{\normalsize Jiangxi Normal  University, Nanchang
330022, P. R. China}
}

\date{}

\maketitle

\begin{abstract}
We consider a free boundary tumor model under the presence of angiogenesis and time delays in the process of proliferation, in which the cell location is incorporated. It is assumed that the tumor attracts blood vessels at a rate proportional to $\alpha$, and a parameter $\mu$ is proportional to the `aggressiveness' of the tumor. In this paper, we first prove that there exists a unique radially symmetric stationary solution $\left(\sigma_{*}, p_{*}, R_{*}\right)$ for all positive $\alpha$, $\mu$. Then a threshold value $\mu_\ast$ is found such that the radially symmetric stationary solution is linearly stable if $\mu<\mu_\ast$ and linearly unstable if $\mu>\mu_\ast$. Our results indicate that the increase of the angiogenesis parameter $\alpha$ would result in the reduction of the threshold value $\mu_\ast$; adding the time delay would not alter the threshold value $\mu_\ast$, but result in a larger stationary tumor, and the larger the tumor aggressiveness parameter $\mu$ is, the greater impact of time delay would have on the size of the stationary tumor.
\\
\\
{Keywords:} Tumor model; Free boundary problem; Time delay; Angiogenesis; Stability
\\
{2020 Mathematics Subject Classification:} 35R35, 35K57, 35B35
\end{abstract}

\section{Introduction}

The first reaction-diffusion mathematical model of tumor growth in the form of a free boundary problem of a system of partial differential equations was proposed in 1972 by
Greenspan \cite{Gr-72,Gr-76}. Since then, a variety of mathematical models have been established from different aspects to describe tumor growth process; see \cite{By-97,BC-95,BC-96,Cui-02,CX-07,FL-15,HZH-19(2),Xu-20,XZB-15,XF-11,ZH-1,ZEC-08,ZC-09,ZC-18}, the reviewing articles \cite{AM,F-07,Low} and references therein. Rigorous mathematical analysis of such free boundary problems was initiated by Friedman and Reitich \cite{FR-99} in 1999, and has made great progress over the past two decades; see \cite{Cui-02,CX-07,CZ-18,FB-03,FH-06(2),FH-06,FH-08,Xu-20,XZB-15,FF-03,FR-01,FL-15,HZH-17,HZH-19(1),HZH-19(2),
HZH-21,XF-11,Xu-10,ZH-1,ZH-2,ZEC-08,ZC-09,ZC-18,ZC-19} and references therein. All these studies may not only provide a better and deeper understanding of the dynamics of tumor growth, but also assist in the treatment of cancer.

In this paper, we consider a mathematical model for tumor growth with time delays and angiogenesis:
\begin{align}
&c\sigma_t-\Delta\sigma+\sigma =0, && x\in\Omega(t),~ t>0,  \label{m-10}
\\
&-\Delta p=\mu[\sigma(\xi(t-\tau;x,t),t-\tau)-\tilde{\sigma}], && x\in\Omega(t), ~t>0,  \label{m-11}
\\
& \left\{
    \begin{aligned}
&  \frac{d \xi}{d s }=-\nabla p(\xi,s), && t-\tau \leq s \leq t, \\
&  \xi=x, && s=t,
    \end{aligned}\label{m-12}
    \right.
\\
&\frac{\partial \sigma}{\partial\vec{n}}+\alpha(\sigma-\bar{\sigma})=0, && x\in\partial\Omega(t), ~t>0,  \label{m-13}
\\
&p=\kappa, && x\in\partial\Omega(t),~ t>0,  \label{m-14}
\\
&V_n=-\pd p{\vec{n}}, && x\in\partial\Omega(t),~ t>0,  \label{m-15}
\\
&\Omega(t)=\Omega_0, && -\tau \leq t \leq 0, \label{m-16}
\\
&\sigma(x,t)=\sigma_0(x), && x\in \Omega_0, ~-\tau \leq t \leq 0, \label{m-17}
\end{align}
where $\Omega(t)\subset\mathbb{R}^2$ denotes the tumor domain at time $t$, $\sigma$ and $p$ are the nutrient concentration and the pressure within the tumor respectively, $c=T_{\rm diffusion}$
$/T_{\rm growth}$ is the ratio of the nutrient diffusion time scale to the tumor growth (e.g., tumor doubling) time scale; thus it is very small and can sometimes be set to be 0 (quasi-steady state approximation). The parameter $\mu$ represents the ``aggressiveness'' of the tumor, $\tilde{\sigma}$ is a threshold concentration for proliferating, $\vec{n}$ is the outward normal, $\alpha$ is a positive constant reflecting the rate at which the tumor attracts blood vessels, $\bar{\sigma}$ is the nutrient concentration outside the tumor, $\kappa$ is the mean curvature and $V_n$ is the velocity of the free boundary in the direction $\vec{n}$. The time delay $\tau$ in \eqref{m-11} represents the time taken for cells to undergo replication (approximately 24 hours). The boundary conditions (\ref{m-13}) and (\ref{m-14}) describe nutrient supply and cell-to-cell adhesiveness at the boundary respectively. \eqref{m-16} and \eqref{m-17} are initial conditions, implying that the initial data are time independent on
$[-\tau,0]$.

The function $\xi(s ; x, t)$ represents the cell location at time $s$ as cells are moving with the velocity field $\vec{V}$, and satisfies
\begin{eqnarray}\label{m-3}
& \left\{
    \begin{aligned}
&\frac{d \xi}{d s }=\vec{V}(\xi, s), \quad t-\tau \leq s \leq t, \\
&\xi|_{s=t}=x.
    \end{aligned}
    \right.
\end{eqnarray}
In other words, $\xi$ tracks the path of the cell currently located at $x$. Hence, if the tumor is assumed to have the structure of porous medium where Darcy's law $\vec{V}=-\nabla p$ holds, then \eqref{m-3} reduces to (\ref{m-12}). For more details we refer to \cite{ZH-1,ZH-2}.

The model \eqref{m-10}--\eqref{m-17} without time delays was studied in \cite{FL-15,HZH-17,HZH-19(1),HZH-21}, where the existence and asymptotic stability of stationary solutions are established; see  \cite{CZ-18,ZC-18,ZC-19} for the general case when the nutrient consumption rate and the cell proliferation rate are not linear. However, in reality, every process, whether it is long or short, would consume time. Thus, compared with those where $\tau=0$, models with time delays are more accurate and consistent with real life. It was Byrne \cite{By-97} who proposed the first free-boundary spherically symmetric tumor model with time delays and the Dirichlet boundary condition
\begin{equation}
\label{m-13-D}
\sigma=\bar\sigma,
\end{equation}
and performed numerical simulations. Rigorous mathematical theoretical analysis of such kinds of spherically symmetric tumor models was made in \cite{FB-03,CX-07,Xu-10,XF-11,XZB-15} and the references therein. Recently, Zhao and Hu \cite{ZH-1} established the non-radially symmetric model \eqref{m-10}--\eqref{m-17} where instead of \eqref{m-13}, \eqref{m-13-D} is prescribed on the tumor boundary. Assuming $c=0$, they investigated the linear stability of the radially symmetric stationary solution under non-radial perturbations and the impact of time delays in the model; bifurcation analysis was done in \cite{ZH-2}.
The radially symmetric version of \eqref{m-10}--\eqref{m-17} was also recently studied by Xu in \cite{Xu-20}. It was proved that the stationary solution is always stable with respect to all radially symmetric perturbations in the case $c=0$.

Motivated by the works mentioned above, in the present paper, we study the non-radially symmetric model \eqref{m-10}--\eqref{m-17}. Besides time delays in the process of proliferation, we work with the boundary
condition \eqref{m-13} stemming from angiogenesis, which is more reasonable compared with \eqref{m-13-D}
from the point of view of biology. In fact, as explained in \cite{FL-15,HZH-21}, the positive constant $\alpha$ in \eqref{m-13} reflects the strength of the blood vessel system of the tumor; the smaller $\alpha$ is, the weaker the blood vessel system of the tumor will be;
$\alpha=0$ means that the tumor does not have its own blood vessel system and $\alpha=\infty$ indicates that the tumor is all surrounded by the blood vessels which reduces to the Dirichlet boundary condition \eqref{m-13-D}.

Our main aim is to discuss linear stability for the model \eqref{m-10}--\eqref{m-17} in the quasi-steady state case $c=0$, and reveal how a tumor's growth dynamics are affected by time delays and angiogenesis. The first main result concerns the existence of radially symmetric stationary solutions.

\begin{theorem}
\label{mainth-t1}
For the time delay $\tau$ small enough, the problem (\ref{m-10})--(\ref{m-17}) admits a unique classical radially symmetric stationary solution $(\sigma_\ast, p_\ast, R_\ast)$.
\end{theorem}

Next, in order to tackle the linear stability, we assume that the initial conditions are perturbed as follows:
\begin{equation}\label{m-18}
\partial \Omega(t):~ r=R_\ast+\varepsilon \rho_{0}(\theta), \quad \sigma(r, \theta, t)=\sigma_\ast(r)+\varepsilon w_{0}(r, \theta), \quad-\tau \leq t \leq 0.
\end{equation}
Substituting
\begin{align}
&\partial \Omega(t):~ r=R_\ast+\varepsilon \rho(\theta, t)+O(\varepsilon^{2}),
\label{m-19-1}\\
&\sigma(r, \theta, t)=\sigma_\ast(r)+\varepsilon w(r, \theta, t)+O(\varepsilon^{2}),
\label{m-19-2}\\
&p(r, \theta, t)=p_\ast(r)+\varepsilon q(r, \theta, t)+O(\varepsilon^{2}),
\label{m-19-3}
\end{align}
into (\ref{m-10})--(\ref{m-17}) and collecting the $\varepsilon$-order terms, we obtain the linearization of (\ref{m-10})--(\ref{m-17}) about the radially symmetric stationary solution $(\sigma_\ast, p_\ast, R_\ast)$.
The second main result of this paper is the following theorem:

\begin{theorem}
\label{mainth-t2}
There exists a threshold value $\mu_\ast>0$ such that if $\mu<\mu_\ast$, then the radially symmetric stationary solution $(\sigma_\ast, p_\ast, R_\ast)$ of \eqref{m-10}--\eqref{m-17} is linearly stable in the sense that
\begin{equation}\label{m-20}
|\rho(\theta, t)-(a_{1} \cos\theta+b_{1} \sin\theta)|\leq Ce^{-\delta t},\quad t>t_0
\end{equation}
for some constants $a_{1}$, $b_{1}$ and positive constants $C$, $\delta$, $t_0$, while if $\mu>\mu_\ast$, then this stationary solution is linearly unstable.
\end{theorem}

\begin{remark}
\label{rem-1}
We stress that $a_1\cos\theta+b_1\sin\theta$ is excluded in \eqref{m-20} and the reason is as follows. The problem (\ref{m-10})--(\ref{m-17}) is invariant under coordinate translations; its stationary solutions are thus not isolated in any function spaces. Hence, to study stability of the radially symmetric stationary solution ensured by Theorem \ref{mainth-t1}, the solutions of (\ref{m-10})--(\ref{m-17}) must be modulated via coordinate translations of $\mathbb{R}^2$.
\end{remark}

\begin{remark}
\label{rem-2}
Let us recall that it was shown in \cite{Xu-20} that stability holds for any $\mu$ under radial perturbations. Now, under non-radial perturbations, we find a finite value $\mu_\ast$ for which the radially symmetric stationary solution changes from stability to instability. The result suggests that larger tumor aggressiveness would induce instability, which is biologically reasonable.
\end{remark}

Now, we present the impact of time delays and  angiogenesis on tumor modeled by the problem  (\ref{m-10})--(\ref{m-17}).
\vspace{2mm}

\noindent{\bf {Conclusion 1}}\quad
 {\it
The increase of the rate of angiogenesis would result in the reduction of the threshold value $\mu_\ast$ for fixed size of the stationary tumor.
}

 \vspace{2mm}

%\noindent{\bf {Conclusion 2}}\quad
% {\it Introducing the time delay into the system would %not alter the threshold value $\mu_\ast$.}

 \vspace{2mm}
\noindent{\bf {Conclusion 2}}\quad
 {\it In contrast to that without time delays, introducing time delays brings the following impacts: (i) it would result in a larger stationary tumor, and the larger the tumor aggressiveness parameter $\mu$ is, the greater impact time delays have on the size of the stationary tumor;
 (ii) it would not alter the threshold value $\mu_\ast$.
 }

 \vspace{2mm}
The structure of this paper is as follows. In Section 2, we collect various results on the modified Bessel functions that are required in the subsequent sections. In Section 3, we establish the existence and uniqueness of the radially symmetric stationary solution $(\sigma_\ast, p_\ast, R_\ast)$. In Section 4, we analyze the linearization of the system about $(\sigma_\ast, p_\ast, R_\ast)$, we give the proofs of Theorem  \ref{mainth-t2}, Conclusions 1 \& 2 by discussing the expansion terms in $\tau$.

\section{Preliminaries}
In this section, we collect some known identities and inequalities for the modified Bessel functions, which will be needed in the latter part of this paper.

The modified Bessel function given by
$$
I_n(r)=\sum_{k=0}^\infty\frac1{k!\Gamma(n+k+1)}\left(\frac{r}2\right)^{n+2k}\quad{\rm for}~n\ge0~{\rm and}~r>0,
$$
satisfies (see \cite{FF-03,FH-06,FR-01})
\begin{align}
&I_n''(r)+\frac1r I_n'(r)-\left(1+\frac{n^2}{r^2}\right)I_n(r)=0,\label{eq2.1}
\\
&I_{n+1}(r)=I_{n-1}(r)-\frac{2n}r I_n(r),
~n\ge1,\label{eq2.2}
\\
&I'_n(r)+\frac{n}r I_n(r)=I_{n-1}(r),~ n\ge1,\label{eq2.3}
\\
&I'_n(r)-\frac{n}r I_n(r)=I_{n+1}(r),~ n\ge0,\label{eq2.4}
\\
&I_{n-1}(r)I_{n+1}(r)<I_n^2(r),~n\ge1,\label{eq2.5}
\\
&I_{n-1}(r)I_{n+1}(r)>I_n^2(r)-\frac2r I_n(r)I_{n+1}(r),~n\ge1,\label{eq2.5-1}
\\
&I_n(r)=\left(\frac1{2\pi r}\right)^{1/2}e^r\left[1-\frac{4n^2-1}{8r}+O(r^{-2})\right]~
{\rm as}~r\to\infty,\label{eq2.9}
\\
&I_m(r)I_n(r)=\sum_{k=0}^\infty\frac{\Gamma(m+n+2k+1)(r/2)^{m+n+2k}}
{k!\Gamma(m+k+1)\Gamma(n+k+1)\Gamma(m+n+k+1)}.\label{eq2.16}
\end{align}
Particularly,
\begin{align}
&\frac{d}{dr}[r I_1(r)]=r I_0(r),\label{eq2.19}
\\
&\frac{d}{dr}[r^2I_0(r)-2r I_1(r)]=r^2I_1(r),\label{eq2.20}
\\
&\frac{d}{dr}\left[\frac12r^2(I^2_1(r)-I^2_0(r))+r I_0(r)I_1(r)\right]=r I^2_1(r).
\label{eq2.21}
\end{align}
Furthermore, if we consider the differential operator
\begin{equation}
\label{Ln}
L_n=-\partial_{rr}-\frac1r\partial_r+\frac{n^2}{r^2},
\end{equation}
then
\begin{align}
&L_1(r[1-2I_2(r)])=2rI_0(r),
\label{eq2.22}
\\
&L_1(r[-I^2_1(r)+I_0(r)I_2(r)])=4I_1(r)\left[I_0(r)-\frac{I_1(r)}{r}\right].
\label{eq2.23}
\end{align}

Now, let us define
\begin{align*}
P_n(r)=\frac{I_{n+1}(r)}{r I_n(r)}\quad{\rm for}~r>0, ~n=0,1,2,\dots.
\end{align*}
Then the following relations can be derived from the preceding properties of the modified Bessel functions:
\begin{align}
&P_n(r)>P_{n+1}(r),\label{eq2.7}
\\
&P_n(r)=\frac1{r^2P_{n+1}(r)+2(n+1)},\quad P_n(0)=\frac1{2n+2},\label{eq2.13}
\\
&P_n'(r)=\frac1r-\frac{2(n+1)}{r}P_n(r)-r P_n^2(r),\label{eq2.14}
\\
&\frac{d}{dr}(r P_0(r))>0,\label{eq2.17}
\\
&\lim_{r\to\infty}P_n(r)=0,\quad\lim_{r\to\infty}r P_n(r)=1.\label{eq2.18}
\end{align}
In addition, we obtain from \citep[(2.15)~and~(2.16)]{FH-06} that
\begin{equation}
P_n'(r)<0.\label{eq2.15}
\end{equation}

\begin{lemma}(\citep[Lemma 2.4]{FR-01})
\label{lem-2.1}
Let $n\ge2$. Then the function
\begin{align*}
G_n(r)=r^2[P_1(r)-P_n(r)]
\end{align*}
satisfies
\begin{align*}
G_n'(r)>0\quad{\rm for~all}~r>0,
\end{align*}
and
\begin{equation}
\label{eq2.11}
\lim_{r\to0^+}G_n(r)=0,\quad\lim_{r\to\infty}G_n(r)=n-1.
\end{equation}
\end{lemma}

We conclude this section by presenting the following result, whose proof is similar to that of Lemma 3.4 in \cite{HZH-17} and so we omit the details here.

\begin{lemma}\label{hl4}
Let $S_1(n)>0$, $S_2(n)>0$ and
\begin{align*}
f_n(r)=S_1(n)P_{n+1}(r)-S_2(n)P_n(r).
\end{align*}
If $f_n(0)>0$, then $f_n(r)>0$ for all $r>0$.
\end{lemma}

\section{Radially symmetric stationary solution}

In this section, we establish the existence and uniqueness of the radially symmetric stationary solution $(\sigma_\ast(r), p_\ast(r), R_\ast)$ to the system
\eqref{m-10}--\eqref{m-17}, which satisfies
\begin{align}
&\sigma''_\ast(r)+\frac1r\sigma'_\ast(r)=\sigma_\ast(r), &&0<r<R_\ast, \label{mn-1}
\\
&p''_\ast(r)+\frac1rp'_\ast(r)=-\mu[\sigma_\ast(\xi_\ast(-\tau; r, 0))-\tilde{\sigma}], &&0<r<R_\ast, \label{mn-2}
\\
&\begin{cases}
      \frac{d\xi_\ast}{ds}(s; r, 0)=-p'_\ast(\xi_\ast(s; r, 0)), & -\tau\leq s\leq 0,\\
      \xi_\ast(s; r, 0)=r, & s=0,
\end{cases}
\label{mn-3}
\\
&\sigma'_\ast(0)=0,\quad \sigma'_\ast(R_\ast)+\alpha(\sigma_\ast(R_\ast)-\bar\sigma)=0, \label{mn-4-1}
\\
&p'_\ast(0)=0,\quad p_\ast(R_\ast)=\frac{1}{R_\ast}, \label{mn-4-2}
\\
&\int^{R_\ast}_0[\sigma_\ast(\xi_\ast(-\tau; r, 0))-\tilde{\sigma}]rdr=0.\label{mn-5}
\end{align}

\begin{proof}[\indent\it\bfseries Proof of Theorem \ref{mainth-t1}]

Introducing the change of variables
$$
\hat r=\frac{r}{R_\ast}, \quad \hat\sigma_\ast(\hat r)=\sigma_\ast(r),\quad \hat p_\ast(\hat r)=R_\ast p_\ast(r), \quad \hat\xi_\ast(s; \hat r, 0)=\frac{\xi_\ast(s; r, 0)}{R_\ast},
$$
the system \eqref{mn-1}--\eqref{mn-5} is reduced, after dropping all hats, to the following system:
\begin{align}
&\sigma''_\ast(r)+\frac1r\sigma'_\ast(r)=R_\ast^2\sigma_\ast(r), &&0<r<1,  \label{mn-7}
\\
&p''_\ast(r)+\frac1rp'_\ast(r)=-\mu R_\ast^3[\sigma_\ast(\xi_\ast(-\tau; r, 0))-\tilde{\sigma}], &&0<r<1,\label{mn-8}
\\
&\begin{cases}
      \frac{d\xi_\ast}{ds}(s; r, 0)=-\frac1{R_\ast^3}p'_\ast(\xi_\ast(s; r, 0)), & -\tau\leq s\leq 0,\\
      \xi_\ast(s; r, 0)=r, & s=0,
\end{cases}
\label{mn-6}
\\
&\sigma'_\ast(0)=0,\quad \sigma'_\ast(1)+\alpha R_\ast(\sigma_\ast(1)-\bar\sigma)=0, \label{eq3.1}
\\
&p'_\ast(0)=0,\quad p_\ast(1)=1, \label{eq3.2}
\\
&\int^1_0[\sigma_\ast(\xi_\ast(-\tau; r, 0))-\tilde{\sigma}]rdr=0.  \label{mn-9}
\end{align}

Using \eqref{eq2.1}, a unique solution of \eqref{mn-7} and \eqref{eq3.1} can be given explicitly by
$$
\sigma_\ast(r)=\frac{\alpha\bar\sigma}{\alpha+R_\ast P_0(R_\ast)}\frac{I_0(R_\ast r)}{I_0(R_\ast)},\quad 0<r<1.
$$
In the sequel, the proof of the existence and uniqueness of $p_\ast(r)$ and $R_\ast$ is quite similar to that of the problem with Dirichlet boundary condition \eqref{m-13-D} (see \cite{ZH-1}), we only need to verify that the function
\begin{equation}
\label{eq3.4}
F(R,0)=\frac{\alpha\bar\sigma}{\alpha+R P_0(R)} P_0(R)-\frac{\tilde\sigma}2
\end{equation}
is monotone decreasing and admits a unique positive zero point for any $\alpha>0$ and any $0<\tilde\sigma<\bar\sigma$, which is ensured by
\eqref{eq2.13} and \eqref{eq2.17}-\eqref{eq2.15}. Hence, for brevity we omit the details here. The proof is complete.
\end{proof}

\begin{remark}
\label{rem-4}
It is worth noting that \eqref{mn-2} together with \eqref{mn-5} implies that
\begin{equation}
\label{eq3.3}
p'_\ast(R_\ast)=0.
\end{equation}
Furthermore, one can obtain from \eqref{mn-3} and \eqref{eq3.3} that
$$
\xi_\ast(s;R_\ast,0)\equiv R_\ast, 0\le\xi_\ast(s;r,0)\le R_\ast,\quad -\tau\le s\le 0,\quad 0\le r\le R_\ast.
$$
\end{remark}

\begin{remark}
\label{rem-5}
It should be pointed out that our stationary solution differs from the classical one. As explained in \cite{ZH-1}, although the free boundary does not move in time, the velocity field inside the stationary tumor is not zero, and movements are necessary to replace dead cells with new daughter cells to reach an equilibrium; because of the time delay, such replacement requires a time $\tau$ for the mitosis to complete and for the daughter cells to move into the right place, reflected by \eqref{mn-3}. Thus, the delay-time derivative can not be set to be zero even for our stationary solution.
\end{remark}

\section{Linear stability}

In this section, we first study the linearization of the problem \eqref{m-10}--\eqref{m-17} with $c=0$ about the radially symmetric stationary solution $(\sigma_\ast,p_\ast,R_\ast)$, and then give the proof of Theorem \ref{mainth-t2}. Meanwhile, the effects of time delay and angiogenesis on the stability and the size of the stationary tumor are discussed.

We shall use the notations \eqref{m-18}--\eqref{m-19-3}, and $\vec{e}_r$, $\vec{e}_\theta$ for the unit vectors in $r$, $\theta$
directions, respectively. Then, written in the rectangular coordinates in $\mathbb{R}^2$,
$$
\vec{e}_r=(\cos\theta,\sin\theta)^\tau,\quad \vec{e}_\theta=(-\sin\theta,\cos\theta)^\tau.
$$
Noticing that the cell location function $\xi(s;r,\theta,t)$ is taken into account here, if
$\xi_1(s;r,\theta,t)$, $\xi_2(s;r,\theta,t)$ are used to denote the polar radius and angle of point $\xi$, respectively, then
\begin{equation}
\label{xi}
\xi(s;r,\theta,t)=\xi_1(s; r, \theta, t)\vec{e}_r(\xi)=\xi_1(s; r, \theta, t)(\cos\xi_2(s;r,\theta,t),\sin\xi_2(s;r,\theta,t))^\tau.
\end{equation}
Expand $\xi_1$, $\xi_2$ in $\varepsilon$ as
\begin{eqnarray}
\label{a-4}
  & \left\{
    \begin{aligned}
&\xi_1=\xi_{10}+\varepsilon\xi_{11}+O(\varepsilon^2), \\
&\xi_2=\xi_{20}+\varepsilon\xi_{21}+O(\varepsilon^2),
\end{aligned}
    \right.
\end{eqnarray}
and it then follows from \eqref{m-12} and \eqref{m-19-3} that
\begin{equation}
\label{xi-10}
\begin{cases}
&\frac{d\xi_{10}}{ds}=-\frac{\partial p_\ast}{\partial r}(\xi_{10}),\quad t-\tau\le s\le t,\\
&\xi_{10}\big|_{s=t}=r;
\end{cases}
\end{equation}
\begin{equation}
\label{xi-11}
\begin{cases}
&\frac{d\xi_{11}}{ds}=-\frac{\partial^2 p_\ast}{\partial r^2}(\xi_{10})\xi_{11}
-\frac{\partial q}{\partial r}(\xi_{10},\xi_{20},s),\quad t-\tau\le s\le t,\\
&\xi_{11}\big|_{s=t}=0;
\end{cases}
\end{equation}
\begin{equation}
\label{xi-20}
\begin{cases}
&\frac{d\xi_{20}}{ds}=0,\quad t-\tau\le s\le t,\\
&\xi_{20}\big|_{s=t}=\theta;
\end{cases}
\end{equation}
\begin{equation}
\label{xi-21}
\begin{cases}
&\frac{d\xi_{21}}{ds}=-\frac1{\xi_{10}^2}\frac{\partial q}{\partial\theta}(\xi_{10},\xi_{20},s),\quad t-\tau\le s\le t,\\
&\xi_{21}\big|_{s=t}=0.
\end{cases}
\end{equation}
One can easily find that $\xi_{20}\equiv\theta$ and the problem \eqref{xi-10} for $\xi_{10}$ is the same as \eqref{mn-3} for $\xi_\ast$ in the radially symmetric case; thus $\xi_{10}$ is independent of $\theta$.

We further substitute \eqref{m-19-1}-\eqref{m-19-3}, \eqref{xi}--\eqref{xi-21} into \eqref{m-10}, \eqref{m-11}, \eqref{m-13}--\eqref{m-15} and collect the $\varepsilon$-order terms,
to obtain the linearized system (see \cite{HZH-19(1),ZH-1})
\begin{align*}
&\Delta w=w, & 0<r<R_\ast, ~ t>0,
\\
&\pd wr(R_\ast, \theta, t)+\alpha w(R_\ast, \theta, t)=-\left(\frac{\partial^2\sigma_\ast}{\partial r^2}+\alpha\frac{\partial \sigma_\ast}{\partial r}\right)\bigg|_{r=R_\ast}\rho(\theta,t),
\\
&\Delta q=-\mu\left[\frac{\partial\sigma_\ast}{\partial r}(\xi_{10}(t-\tau; r, t))\xi_{11}(t-\tau; r, \theta, t)+w(\xi_{10}(t-\tau; r, t), \theta, t-\tau)\right], & 0<r<R_\ast, ~ t>0,
\\
&q(R_\ast, \theta, t)=-\frac{1}{R_\ast^2}[\rho(\theta,t)+\rho_{\theta\theta}(\theta,t)],
\\
&\frac{d\rho}{dt}=-\frac{\partial^2 p_\ast}{\partial r^2}\bigg|_{r=R_\ast}\rho(\theta,t)-\frac{\partial q}{\partial r}\bigg|_{r=R_\ast}, & t>0.
\end{align*}

Now, we proceed to seek solutions of the form
\begin{align*}
&w(r,\theta,t)=w_{n}(r,t)\cos(n\theta),
\\
&q(r,\theta,t)=q_{n}(r,t)\cos(n\theta),
\\
&\rho(\theta,t)=\rho_{n}(t)\cos(n\theta),
\\
&\xi_{11}(s;r,\theta,t)=\varphi_{n}(s;r,t)\cos(n\theta).
\end{align*}
Similarly, we can also seek solutions of the form
\begin{align*}
&w(r, \theta, t)=w_{n}(r, t)\sin(n\theta),
\\
&q(r, \theta, t)=q_{n}(r, t)\sin(n\theta),
\\
&\rho(\theta, t)=\rho_{n}(t)\sin(n\theta),
\\
&\xi_{11}(s; r, \theta, t)=\varphi_{n}(s; r, t)\sin(n\theta).
\end{align*}
Here, since the concrete expression for $\xi_{21}$ is not needed, we will not write it in detail.
Thus, using the relation
$$
\Delta=\partial_{rr}+\frac{1}{r}\partial_{r}+\frac{1}{r^2}\partial_{\theta\theta},
$$
we get
\begin{align}
&-\Delta w_n+\left(\frac{n^2}{r^2}+1\right)w_n=0,\qquad\qquad\qquad\qquad\qquad\qquad\qquad\qquad\qquad\qquad 0<r<R_\ast, \label{a-14}
\\
&\frac{\partial w_n}{\partial r}(R_\ast, t)+\alpha w_n(R_\ast, t)=-\left(\frac{\partial^2 \sigma_\ast}{\partial r^2}+\alpha\frac{\partial \sigma_\ast}{\partial r}\right)\bigg|_{r=R_\ast}\rho_n(t), \label{a-15}
\\
&-\Delta q_n+\frac{n^2}{r^2}q_n=\mu\left[w_n(\xi_{10}(t-\tau; r, t), t-\tau)
+\frac{\partial\sigma_\ast}{\partial r}(\xi_{10}(t-\tau; r, t))\varphi_n(t-\tau; r, t)\right], ~0<r<R_\ast,
\label{a-16}\\
&q_n(R_\ast, t)=\frac{n^2-1}{R^2_\ast}\rho_n(t), \label{a-17}
\\
&\frac{d\rho_n(t)}{dt}=-\frac{\partial^2 p_\ast}{\partial r^2}\bigg|_{r=R_\ast}\rho_n(t)-\frac{\partial q_n}{\partial r}\bigg|_{r=R_\ast}, \label{a-18}
\\
&\frac{d\varphi_n(s;r,t)}{ds}=-\frac{\partial^2 p_\ast}{\partial r^2}(\xi_{10}(s;r,t))\varphi_n(s;r, t)-\frac{\partial q_n}{\partial r}(\xi_{10}(s;r,t),s), \qquad\qquad t-\tau\leq s\leq t, \label{a-19}
\\
&\varphi_n\big|_{s=t}=0.\label{a-19-1}
\end{align}

\subsection{Expansion in $\tau$}

It may be impossible to solve the system \eqref{mn-1}-\eqref{mn-5}, \eqref{a-14}-\eqref{a-19-1} explicitly. Thus, in order to study the impact of $\tau$ on this system, we write
\begin{align*}
R_\ast&=R_\ast^0+\tau R_\ast^1+O(\tau^2), \\
\sigma_\ast&=\sigma_\ast^0+\tau\sigma_\ast^1+O(\tau^2), \\
p_\ast&=p_\ast^0+\tau p_\ast^1+O(\tau^2), \\
w_n&=w_n^0+\tau w_n^1+O(\tau^2), \\
q_n&=q_n^0+\tau q_n^1+O(\tau^2), \\
\rho_n&=\rho_n^0+\tau\rho_n^1+O(\tau^2),
\end{align*}
which is reasonable because the time delay $\tau$ is actually very small.
Substituting these expansions into the system, as in \cite{ZH-1}, one can obtain two separate systems. One for all zeroth-order terms in $\tau$ is
\begin{align}
\label{eq4.12}
&\frac{\partial^2\sigma^0_\ast}{\partial r^2}+\frac{1}{r}\frac{\partial\sigma^0_\ast}{\partial r}=\sigma^0_\ast, \quad 0<r<R^0_\ast, &&\frac{\partial\sigma^0_\ast}{\partial r}(R^0_\ast)+\alpha[\sigma^0_\ast(R^0_\ast)-\bar{\sigma}]=0,
\\
\label{a-53}
&-\frac{\partial^2p^0_\ast}{\partial r^2}-\frac{1}{r}\frac{\partial p^0_\ast}{\partial r}=\mu(\sigma^0_\ast-\tilde{\sigma}), \quad 0<r<R^0_\ast,&&
p^0_\ast(R^0_\ast)=\frac{1}{R^0_\ast},
\\
\label{eq4.13}
&\int^{R^0_\ast}_0[\sigma^0_\ast(r)-\tilde\sigma]rdr=0,&&
\\
\label{a-54}
&-\frac{\partial^2\omega^0_n}{\partial r^2}-\frac{1}{r}\frac{\partial\omega^0_n}{\partial r}+\left(\frac{n^2}{r^2}+1\right)\omega^0_n=0, \quad 0<r<R^0_\ast,
&&
\frac{\partial\omega^0_n}{\partial r}(R^0_\ast, t)+\alpha\omega^0_n(R^0_\ast, t)=-\lambda\rho^0_n(t),
\\
\label{a-55}
&-\frac{\partial^2q^0_n}{\partial r^2}-\frac{1}{r}\frac{\partial q^0_n}{\partial r}+\frac{n^2}{r^2}q^0_n=\mu\omega^0_n, \quad 0<r<R^0_\ast, &&
q^0_n(R^0_\ast, t)=\frac{n^2-1}{(R^0_\ast)^2}\rho^0_n(t),
\\
\label{a-56}
&\frac{d\rho^0_n(t)}{dt}=-\frac{\partial^2p^0_\ast}{\partial r^2}(R^0_\ast)\rho^0_n(t)-\frac{\partial q^0_n}{\partial r}(R^0_\ast,t),&&
\end{align}
where
$$
\lambda=\left(\frac{\partial^2\sigma^0_\ast}{\partial r^2}+\alpha\frac{\partial\sigma^0_\ast}{\partial r}\right)\bigg|_{r=R^0_\ast}.
$$
The other for all first-order terms in $\tau$ is the following:
\begin{align}
&\frac{\partial^2\sigma^1_\ast}{\partial r^2}+\frac{1}{r}\frac{\partial\sigma^1_\ast}{\partial r}=\sigma^1_\ast, \quad 0<r<R^0_\ast,
\quad \frac{\partial\sigma^1_\ast}{\partial r}(R^0_\ast)+\alpha\sigma^1_\ast(R^0_\ast)=-\lambda R^1_\ast,
\label{eq4.14}
\\
&-\frac{\partial^2p^1_\ast}{\partial r^2}-\frac{1}{r}\frac{\partial p^1_\ast}{\partial r}=\mu\left(\frac{\partial \sigma^0_\ast}{\partial r}\frac{\partial p^0_\ast}{\partial r}+\sigma^1_\ast\right), \quad 0<r<R^0_\ast,
\quad p^1_\ast(R^0_\ast)=-\left[\frac1{(R^0_\ast)^2}+\frac{\partial p^0_\ast}{\partial r}(R^0_\ast)\right]R^1_\ast, \label{a-82}
\\
&[\sigma^0_\ast(R^0_\ast)-\tilde\sigma]R^0_\ast R^1_\ast+\int^{R^0_\ast}_0\left[\frac{\partial\sigma^0_\ast}{\partial r}(r)\frac{\partial p^0_\ast}{\partial r}(r)+\sigma^1_\ast(r)\right]rdr=0,\label{a-81}
\\
&-\frac{\partial^2\omega^1_n}{\partial r^2}-\frac{1}{r}\frac{\partial\omega^1_n}{\partial r}+\left(\frac{n^2}{r^2}+1\right)\omega^1_n=0, \quad 0<r<R^0_\ast, \label{a-83}
\\
&\frac{\partial\omega^1_n}{\partial r}(R^0_\ast, t)+\alpha\omega^1_n(R^0_\ast, t)=-\left[\frac{\partial^2\omega^0_n}{\partial r^2}(R^0_\ast, t)+\alpha\frac{\partial\omega^0_n}{\partial r}(R^0_\ast, t)\right]R^1_\ast-\lambda\rho^1_n(t)\nonumber
\\
&~~~~~~~~~~~~~~~~~~~~~~~~~~
-\left[\frac{\partial^3 \sigma^0_\ast}{\partial r^3}(R^0_\ast)R^1_\ast+\alpha\frac{\partial^2\sigma^0_\ast}{\partial r^2}(R^0_\ast)R^1_\ast+\frac{\partial^2\sigma^1_\ast}{\partial r^2}(R^0_\ast)+\alpha\frac{\partial \sigma^1_\ast}{\partial r}(R^0_\ast)\right]\rho^0_n(t),
\label{a-83-1}
\\
&-\frac{\partial^2q^1_n}{\partial r^2}-\frac{1}{r}\frac{\partial q^1_n}{\partial r}+\frac{n^2}{r^2}q^1_n=\mu\left(\frac{\partial\sigma^0_\ast}{\partial r}\frac{\partial q^0_n}{\partial r}+\frac{\partial \omega^0_n}{\partial r}\frac{\partial p^0_\ast}{\partial r}-\frac{\partial\omega^0_n}{\partial t}+\omega^1_n\right), \quad 0<r<R^0_\ast, \label{a-84}
\\
&q^1_n(R^0_\ast, t)=-\frac{\partial q^0_n}{\partial r}(R^0_\ast, t)R^1_\ast+\frac{n^2-1}{(R^0_\ast)^2}\rho^1_n(t)-\frac{2(n^2-1)}{(R^0_\ast)^3}R^1_\ast\rho^0_n(t),
\label{a-84-1}
\\
&\frac{d\rho^1_n(t)}{dt}=-\frac{\partial^2p^0_\ast}{\partial r^2}(R^0_\ast)\rho^1_n(t)-\left[\frac{\partial^3p^0_\ast}{\partial r^3}(R^0_\ast)R^1_\ast
+\frac{\partial^2p^1_\ast}{\partial r^2}(R^0_\ast)\right]\rho^0_n(t)-\frac{\partial^2q^0_n}{\partial r^2}(R^0_\ast, t)R^1_\ast-\frac{\partial q^1_n}{\partial r}(R^0_\ast, t).\label{a-85}
\end{align}

\subsection{Zeroth-order terms in $\tau$}

We now solve the system \eqref{eq4.12}--\eqref{a-56} explicitly. Using \eqref{eq2.1}, it is easy to see from \eqref{eq4.12} that
\begin{equation}
\label{a-51}
\sigma^0_\ast(r)=\frac{\alpha\bar{\sigma}}{\alpha+R^0_\ast P_0(R^0_\ast)}\frac{I_0(r)}{I_0(R^0_\ast)}.
\end{equation}
Then by \eqref{eq2.3} and \eqref{eq2.4}, we compute
\begin{align}
\frac{\partial\sigma^0_\ast}{\partial r}(r)=&\frac{\alpha\bar{\sigma}}{\alpha+R^0_\ast P_0(R^0_\ast)}\frac{I_1(r)}{I_0(R^0_\ast)}=\frac{\alpha\bar{\sigma}}{\alpha+R^0_\ast P_0(R^0_\ast)}\frac{rI_0(r)P_0(r)}{I_0(R^0_\ast)},\label{eq4.23}
\\
\frac{\partial^2\sigma^0_\ast}{\partial r^2}(r)=&\frac{\alpha\bar{\sigma}}{\alpha+R^0_\ast P_0(R^0_\ast)}\frac{I_0(r)}{I_0(R^0_\ast)}(1-P_0(r)),\label{eq4.24}
\end{align}
from which it follows that
\begin{equation}
\label{eq4.15}
\lambda=\frac{\alpha\bar{\sigma}}{\alpha+R^0_\ast P_0(R^0_\ast)}[1-P_0(R^0_\ast)+\alpha R^0_\ast P_0(R^0_\ast)].
\end{equation}
Solving \eqref{a-53} gives
\begin{equation}
\label{a-57}
p^0_\ast(r)=\frac{\mu\tilde{\sigma}r^2}{4}-\frac{\mu\alpha\bar{\sigma}}{\alpha+R^0_\ast P_0(R^0_\ast)}\frac{I_0(r)}{I_0(R^0_\ast)}+\frac{\mu\alpha\bar{\sigma}}{\alpha+R^0_\ast P_0(R^0_\ast)}+\frac{1}{R^0_\ast}-\frac{\mu\tilde{\sigma}(R^0_\ast)^2}{4}.
\end{equation}
Substituting \eqref{a-51} into \eqref{eq4.13} and using \eqref{eq2.19}, we derive
\begin{equation}
\label{a-52}
\frac{\tilde{\sigma}}{2\bar{\sigma}}=\frac{\alpha P_0(R^0_\ast)}{\alpha+R^0_\ast P_0(R^0_\ast)}.
\end{equation}
Being similar to \eqref{eq3.4}, for any $\alpha>0$ and any $0<\tilde\sigma<\bar\sigma$, the equation \eqref{a-52} admits a unique positive root $R^0_\ast$.

In view of \eqref{eq2.1}, we compute from \eqref{a-54} that
\begin{equation}
\label{a-58}
\omega^0_n(r,t)=\frac{-\lambda}{\alpha+h_n(R^0_\ast)}\frac{I_n(r)}{I_n(R^0_\ast)}\rho^0_n(t),
\end{equation}
where
\begin{equation}
\label{eq4.17}
h_n(x)=\frac n{x}+xP_n(x),\quad x>0.
\end{equation}
To find $q^0_n$, let $\eta^0_n=q^0_n+\mu\omega^0_n$. Combining \eqref{a-54} and \eqref{a-55}, we then find that $\eta^0_n$ satisfies
\begin{equation*}
L_n\eta^0_n=0, \quad 0<r<R^0_\ast,
\end{equation*}
which implies
$$
\eta^0_n(r,t)=C_1(t)r^n,
$$
where the operator $L_n$ is defined by \eqref{Ln} and $C_1(t)$ is an unknown function to be determined later.
As a result,
\begin{equation}
\label{a-62}
q^0_n(r,t)=\eta^0_n(r,t)-\mu\omega^0_n(r,t)=C_1(t)r^n-\mu\omega^0_n(r,t).
\end{equation}
A combination of \eqref{a-58}, \eqref{a-62} and the boundary value condition in \eqref{a-55} gives
\begin{equation*}
C_1(t)=\frac{1}{(R^0_\ast)^n}\bigg[\frac{n^2-1}{(R^0_\ast)^2}-\frac{\mu\lambda}{\alpha+h_n(R^0_\ast)}\bigg]\rho^0_n(t),
\end{equation*}
and so
\begin{equation}
\label{a-63}
q^0_n(r,t)=-\mu\omega^0_n(r,t)+
\bigg[\frac{n^2-1}{(R^0_\ast)^2}-\frac{\mu\lambda}{\alpha+h_n(R^0_\ast)}\bigg]\frac{r^n}{(R^0_\ast)^n}\rho^0_n(t).
\end{equation}
According to \eqref{a-56}, before solving for $\rho^0_n(t)$, it is necessary to compute from \eqref{a-57} and \eqref{a-58}
that
\begin{align}
\frac{\partial p^0_\ast}{\partial r}(r)=&\frac{\mu\tilde{\sigma}r}{2}-\frac{\mu\alpha\bar{\sigma}}{\alpha+R^0_\ast P_0(R^0_\ast)}\frac{I_1(r)}{I_0(R^0_\ast)},\label{ab-1}
\\
\frac{\partial^2 p^0_\ast}{\partial r^2}(r)=&\frac{\mu\tilde{\sigma}}{2}-\frac{\mu\alpha\bar{\sigma}}{\alpha+R^0_\ast P_0(R^0_\ast)}\frac{I_0(r)}{I_0(R^0_\ast)}[1-P_0(r)],\label{eq4.19}
\\
\frac{\partial\omega^0_n}{\partial r}(r,t)=&\frac{-\lambda}{\alpha+h_n(R^0_\ast)}\frac{I_n(r)}{I_n(R^0_\ast)}h_n(r)\rho^0_n(t),
\label{eq4.18}
\end{align}
which, combined with \eqref{eq2.13}, \eqref{eq4.15}, \eqref{a-52}, \eqref{a-58} and \eqref{a-63}, implies
\begin{align}
\frac{\partial^2 p^0_\ast}{\partial r^2}(R^0_\ast)=&\frac{\mu\alpha\bar{\sigma}}{\alpha+R^0_\ast P_0(R^0_\ast)}
[2P_0(R^0_\ast)-1]=-\frac{\mu\alpha\bar{\sigma}}{\alpha+R^0_\ast P_0(R^0_\ast)}(R^0_\ast)^2P_0(R^0_\ast)P_1(R^0_\ast),
\label{a-60}
\\
\frac{\partial q^0_n}{\partial r}(R^0_\ast,t)=&\left[\frac{n(n^2-1)}{(R^0_\ast)^3}
+\mu\frac{\alpha\bar{\sigma}[1-P_0(R^0_\ast)+\alpha R^0_\ast P_0(R^0_\ast)]}{\alpha+R^0_\ast P_0(R^0_\ast)}\frac{R^0_\ast P_n(R^0_\ast)}{\alpha+h_n(R^0_\ast)}\right]\rho^0_n(t).
\label{eq4.20}
\end{align}
We insert \eqref{a-60}, \eqref{eq4.20} into \eqref{a-56} and use \eqref{eq2.13} again, to arrive at
\begin{align}
\frac{d\rho^0_n(t)}{dt}=&\left(-\frac{n(n^2-1)}{(R^0_\ast)^3}
+\mu\frac{\alpha\bar\sigma R^0_\ast P_0(R^0_\ast)}{\alpha+R^0_\ast P_0(R^0_\ast)}
\frac{nP_1(R^0_\ast)-P_n(R^0_\ast)+\alpha R^0_\ast[P_1(R^0_\ast)-P_n(R^0_\ast)]}{h_n(R^0_\ast)+\alpha}\right)\rho^0_n(t),\nonumber
\\
=&[-A_n(R^0_\ast)+\mu B_n(R^0_\ast,\alpha)]\rho^0_n(t),\label{eq4.10}
\end{align}
where
\begin{align}
&A_n(R^0_\ast)=\frac{n(n^2-1)}{(R^0_\ast)^3},\label{eq4.1}
\\
&B_n(R^0_\ast,\alpha)=\frac{\alpha\bar\sigma R^0_\ast P_0(R^0_\ast)}{\alpha+R^0_\ast P_0(R^0_\ast)}
\frac{nP_1(R^0_\ast)-P_n(R^0_\ast)+\alpha R^0_\ast[P_1(R^0_\ast)-P_n(R^0_\ast)]}{h_n(R^0_\ast)+\alpha}.
\label{eq4.2}
\end{align}
Therefore,
\begin{equation}
\label{a-66}
        \rho^0_n(t)=\rho^0_n(0)\exp\left\{[-A_n(R^0_\ast)+\mu B_n(R^0_\ast,\alpha)]t\right\}.
\end{equation}

We now proceed to analyze the asymptotic behavior of $\rho^0_n(t)$ as $t\to\infty$. Evidently, $A_0(R^0_\ast)=A_1(R^0_\ast)=0$, $A_n(R^0_\ast)>0$ for $n\ge2$, $B_1(R^0_\ast,\alpha)=0$ and \eqref{eq2.7} implies $B_0(R^0_\ast,\alpha)<0$, $B_n(R^0_\ast,\alpha)>0$ for $n\ge2$. As a result, we immediately get the next two propositions.

\begin{proposition}
\label{l1}
If $n=0$, then for any $\mu>0$, there exists a positive constant $\delta_1$ such that $|\rho_0^0(t)|\leq|\rho_0^0(0)|e^{-\delta_1 t}$ for all $t>0$.
\end{proposition}

\begin{proposition}
\label{l2}
If $n=1$, then for any $\mu>0$, $\rho_1^0(t)=\rho_1^0(0)$ for all $t>0$.
\end{proposition}

\begin{remark}
\label{mainth-r5}
Proposition \ref{l1} is consistent with the result in Xu \cite{Xu-20}. As a matter of fact, the mode $n=0$ represents radially symmetric perturbations, because in this case,
$$
r=R_\ast+\varepsilon\rho_0(t)=R_\ast^0+\varepsilon\rho_0^0(t)+\tau(R_\ast^1+\varepsilon\rho_0^1(t))+O(\tau^2),
$$
and when $\tau$ is small, we do not expect the first-order terms to have a major contribution.
\end{remark}

If $n\geq2$, we denote
\begin{equation}
\label{a-69}
\mu^0_n(R^0_\ast,\alpha)=\frac{A_n(R^0_\ast)}{B_n(R^0_\ast,\alpha)},
\end{equation}
or equivalently, by \eqref{eq4.17}, \eqref{eq4.1}, \eqref{eq4.2},
\begin{equation}
\label{h-46}
\mu^0_n(R^0_\ast,\alpha)=\frac{\alpha+R^0_\ast P_0(R^0_\ast)}{\alpha\bar{\sigma}(R^0_\ast)^4 P_0(R^0_\ast)}\frac{n(n^2-1)\left[\frac{n}{R^0_\ast}+\alpha+R^0_\ast P_n(R^0_\ast)\right]}{(n+\alpha R^0_\ast)P_1(R^0_\ast)-(1+\alpha R^0_\ast)P_n(R^0_\ast)}.
\end{equation}

\begin{lemma}
\label{l4}
$\mu^0_n(R^0_\ast,\alpha)$ is monotone increasing in $n$ for all $n\ge2$.
\end{lemma}

\begin{proof}
According to \eqref{h-46}, the desired inequality
\begin{equation*}
%\label{eq4.3}
\mu^0_n(R^0_\ast,\alpha)<\mu^0_{n+1}(R^0_\ast,\alpha)
\end{equation*}
is equivalent to
\begin{equation*}
%\label{h-49}
    \begin{aligned}
&(n-1)\left[\frac{n}{R^0_\ast}+\alpha+R^0_\ast P_n(R^0_\ast)\right]
\left[(n+1+\alpha R^0_\ast)P_1(R^0_\ast)-(1+\alpha R^0_\ast)P_{n+1}(R^0_\ast)\right]
\\
<&(n+2)\left[\frac{n+1}{R^0_\ast}+\alpha+R^0_\ast P_{n+1}(R^0_\ast)\right]
\left[(n+\alpha R^0_\ast)P_1(R^0_\ast)-(1+\alpha R^0_\ast)P_n(R^0_\ast)\right].
   \end{aligned}
\end{equation*}
It is then enough to show that
\begin{equation}
\label{h-50}
Q_1(n, R^0_\ast)R^0_\ast\alpha^2+Q_2(n, R^0_\ast)\alpha+\frac{Q_3(n, R^0_\ast)}{R_\ast^0}>0,
\end{equation}
where
\begin{equation}
\label{eq4.3}
Q_1(n, R^0_\ast)=
3P_1(R^0_\ast)-(n+2)P_n(R^0_\ast)+(n-1)P_{n+1}(R^0_\ast),
\end{equation}
\begin{align*}
\begin{split}
Q_2(n, R^0_\ast)=
&(6n+3)P_1(R^0_\ast)-(n+2)^2P_n(R^0_\ast)+(n^2-1)P_{n+1}(R^0_\ast)
\\
&-(n-1)(R^0_\ast)^2P_1(R^0_\ast)P_n(R^0_\ast)+(n+2)(R^0_\ast)^2P_1(R^0_\ast)P_{n+1}(R^0_\ast)
\\
&-3(R^0_\ast)^2P_n(R^0_\ast)P_{n+1}(R^0_\ast),
\end{split}
\end{align*}
\begin{align}
Q_3(n,R^0_\ast)=
&3n(n+1)P_1(R^0_\ast)-(n+1)(n+2)P_n(R^0_\ast)+n(n-1)P_{n+1}(R^0_\ast)\nonumber
\\
&-(n^2-1)\left(R^0_\ast\right)^2P_1(R^0_\ast)P_n(R^0_\ast)
+n(n+2)\left(R^0_\ast\right)^2P_1(R^0_\ast)P_{n+1}(R^0_\ast)\nonumber
\\
&-3\left(R^0_\ast\right)^2P_n(R^0_\ast)P_{n+1}(R^0_\ast).\label{eq4.5}
\end{align}

In the sequel, we shall complete the proof of the lemma via the next three separate lemmas, which tell us that $Q_i(n, R^0_\ast)>0$ $(i=1,2,3)$ for every $n\ge2$, and therefore prove \eqref{h-50}.
\end{proof}

\begin{lemma}\label{hl5}
$Q_1(n, R^0_\ast)>0$ for $n\geq2$.
\end{lemma}

\begin{proof}
By \eqref{eq4.3}, it suffices to show that
\begin{equation}
\label{eq4.6}
(n+2)P_n(x)-(n-1)P_{n+1}(x)-3P_1(x)<0,\quad x>0.
\end{equation}
Notice that \eqref{eq2.13} implies that
$$
P_1(x)=\frac1{x^2P_0(x)}-\frac2{x^2},\quad x>0.
$$
Substituting this into \eqref{eq4.6} yields
\begin{equation*}
(n+2)x^2P_n(x)-(n-1)x^2P_{n+1}(x)-\frac3{P_0(x)}+6<0,\quad x>0,
\end{equation*}
which has been established in \cite{HZH-19(2)} (see (3.31)). The proof is complete.
\end{proof}

\begin{lemma}
\label{hl6}
$Q_2(n, R^0_\ast)>0$ for $n\geq2$.
\end{lemma}

\begin{proof}
Using \eqref{eq2.7}, we have
\begin{align*}
Q_2(n, R^0_\ast)>&(6n+3)P_n(R^0_\ast)-(n+2)^2P_n(R^0_\ast)+(n^2-1)P_{n+1}(R^0_\ast)
\\
&-(n-1)(R^0_\ast)^2P_1(R^0_\ast)P_n(R^0_\ast)
+(n+2)(R^0_\ast)^2P_1(R^0_\ast)P_{n+1}(R^0_\ast)
\\
&-3(R^0_\ast)^2P_n(R^0_\ast)P_{n+1}(R^0_\ast)
\\
=&J_{1,n}(R^0_\ast)+3(R^0_\ast)^2P_1(R^0_\ast)P_n(R^0_\ast)P_{n+1}(R^0_\ast)J_{2,n}(R^0_\ast),
\end{align*}
where
\begin{align*}
J_{1,n}(x)=&(n^2-1)P_{n+1}(x)-(n^2-2n+1)P_n(x),\quad x>0,
\\
J_{2,n}(x)=&\frac{n+2}{3}\frac{1}{P_n(x)}-\frac{n-1}{3}\frac{1}{P_{n+1}(x)}-
\frac{1}{P_1(x)},\quad x>0.
\end{align*}
Since by \eqref{eq2.13},
\begin{align*}
(n^2-1)P_{n+1}(0)-(n^2-2n+1)P_n(0)=\frac{(n-1)(n+3)}{2(n+1)(n+2)}>0\quad{\rm for}~n\ge2,
\end{align*}
applying Lemma \ref{hl4} yields that $J_{1,n}(R^0_\ast)>0$.
We now claim that
\begin{equation}
\label{h-58}
J_{2,n}(x)>0\quad{\rm for}~n\ge2~{\rm and}~x>0.
\end{equation}
In fact, we first apply \eqref{eq2.13} to get $J_{2,n}(0)>0$, which together with its continuity implies that $J_{2,n}(x)>0$ on some interval $[0,\tau_n]$.
Next, an elementary calculation based on \eqref{eq2.14} gives that
$$
J'_{2,n}(x)=W_1(n,x)J_{2,n}(x)+W_2(n,x),
$$
where
\begin{align*}
W_1(n,x)=&\frac1xJ_{2,n}(x)-\frac2x\left[\frac{n+2}3\frac1{P_n(x)}-\frac{n-1}3\frac1{P_{n+1}(x)}\right]+\frac4x,
\\
W_2(n,x)=&\frac{n+2}3\frac{n-1}3\frac1x\left[\frac1{P_n(x)}-\frac1{P_{n+1}(x)}\right]^2
+\frac{2(n-1)}3\frac1x\frac{(n+2)P_{n+1}(x)-nP_n(x)}{P_n(x)P_{n+1}(x)}.
\end{align*}
Noticing that
$$
(n+2)P_{n+1}(0)-nP_n(0)>0,
$$
using Lemma \ref{hl4} and \eqref{eq2.7}, we obtain $W_2(n,x)>0$ for every $x>0$ and every $n\ge2$.
Combining with the fact that
\begin{equation*}
J_{2,n}(x)=\exp\left\{\int^{x}_{\tau_n}W_1(n,t)dt\right\}\left(J_{2,n}(\tau_n)+\int^{x}_{\tau_n}
\exp\left\{\int^{t}_{\tau_n}-W_1(n,s)ds\right\}W_2(n,t)dt\right),\quad x>\tau_n,
\end{equation*}
we arrive at the assertion \eqref{h-58} and therefore complete the proof of this lemma.
\end{proof}

\begin{lemma}\label{hl7}
$Q_3(n,R^0_\ast)>0$ for $n\geq2$.
\end{lemma}

\begin{proof}
\eqref{eq2.7} implies $P_1(R^0_\ast)>P_n(R^0_\ast)$ and we thus obtain from \eqref{eq4.5} that
\begin{align*}
Q_3(n,R^0_\ast)>&3n(n+1)P_n(R^0_\ast)-(n+1)(n+2)P_n(R^0_\ast)+n(n-1)P_{n+1}(R^0_\ast)
\\
&-(n^2-1)\left(R^0_\ast\right)^2P_1(R^0_\ast)P_n(R^0_\ast)
+n(n+2)\left(R^0_\ast\right)^2P_1(R^0_\ast)P_{n+1}(R^0_\ast)
\\
&-3\left(R^0_\ast\right)^2P_1(R^0_\ast)P_{n+1}(R^0_\ast)
\\
=&2(n^2-1)P_n(R^0_\ast)+n(n-1)P_{n+1}(R^0_\ast)
\\
&+(n-1)\left(R^0_\ast\right)^2 P_1(R^0_\ast)[(n+3)P_{n+1}(R^0_\ast)-(n+1)P_n(R^0_\ast)].
\end{align*}
Since
\begin{equation*}
(n+3)P_{n+1}(0)-(n+1)P_n(0)=\frac1{2(n+2)}>0,
\end{equation*}
by Lemma \ref{hl4}, we conclude $Q_3(n,R^0_\ast)>0$ for $n\geq2$. The proof is complete.
\end{proof}

Since Lemmas \ref{l1} and \ref{l2} are valid for all $\mu$, we define $\mu^0_0(R^0_\ast,\alpha)=\mu^0_1(R^0_\ast,\alpha)=\infty$. Set
\begin{equation}
\label{a-71}
\mu_\ast(R^0_\ast,\alpha)=\min\{\mu^0_0(R^0_\ast,\alpha),~ \mu^0_1(R^0_\ast,\alpha),~ \mu^0_2(R^0_\ast,\alpha),~ \mu^0_3(R^0_\ast,\alpha),~ \cdots\}.
\end{equation}
It then follows from Lemma \ref{l4} that
\begin{equation}
\label{a-72}
\mu_\ast(R^0_\ast,\alpha)=\mu^0_2(R^0_\ast,\alpha).
\end{equation}
Furthermore, the following result holds.

\begin{proposition}\label{l5}
For $n\geq2$ and $0<\mu<\mu_\ast(R^0_\ast,\alpha)$, there exists a positive constant $\delta_2$, depending on $\mu$, $R^0_\ast$, $\alpha$ but being independent of $n$, such that
\begin{equation}
\label{a-73}
|\rho^0_n(t)|\le|\rho^0_n(0)|e^{-\delta_2 n^3t}\quad{\rm for~ all}~t>0.
\end{equation}
\end{proposition}

\begin{proof}
By \eqref{eq4.1}, \eqref{a-66}, \eqref{a-69} and \eqref{a-71}, we find
\begin{equation*}
|\rho^0_n(t)|\le|\rho^0_n(0)|\exp\left\{-\left(1-\frac{\mu}{\mu_\ast(R^0_\ast,\alpha)}\right)
\frac{n(n^2-1)}{(R^0_\ast)^3}t\right\}.
\end{equation*}
Noticing that
\begin{equation}
\label{eq4.38}
n^3-n\ge\frac34n^3\quad{\rm for}~n\ge2,
\end{equation}
we further derive
\begin{equation*}
|\rho^0_n(t)|\le|\rho^0_n(0)|\exp\left\{-\frac34\left(1-\frac{\mu}{\mu_\ast(R^0_\ast,\alpha)}\right)
\frac1{(R^0_\ast)^3}n^3t\right\}.
\end{equation*}
Taking
\begin{equation}
\label{delta}
\delta_2=\frac34\left(1-\frac{\mu}{\mu_\ast(R^0_\ast,\alpha)}\right)
\frac1{(R^0_\ast)^3},
\end{equation}
the desired result \eqref{a-73} follows.
\end{proof}

\begin{proof}[\indent\it\bfseries Proof of Theorem \ref{mainth-t2}]
Using Propositions \ref{l1}, \ref{l2} and \ref{l5}, we can derive \eqref{m-20} by an argument quite similar to that in the proof of Theorem 1.1 of \cite{HZH-19(2)}. On the other hand, the linear instability in the case $\mu>\mu_\ast$ follows by taking $n=2$ in \eqref{a-66}. The proof is complete.
\end{proof}

\begin{remark}
\label{rem-4.1}
We now find that when the time delay $\tau$ is sufficiently small and the tumor proliferation intensity
$\mu$ is smaller than the threshold value $\mu_\ast$, the radially symmetric stationary solution
$(\sigma_\ast,p_\ast,R_\ast)$ is linearly stable even under non-radial perturbations, while when $\mu$ is larger than $\mu_\ast$, being different from that in \cite{Xu-20}, $(\sigma_\ast,p_\ast,R_\ast)$ is no longer stable under non-radial perturbations and the unstable mode comes from mode $2$.
\end{remark}

Before proceeding further, since the threshold value $\mu_\ast$ depends on the parameter $\alpha$, which represents the rate of angiogenesis, it would be interesting to find out the effect of angiogenesis on $\mu_\ast$. Precisely speaking, we get the following result.

\begin{lemma}
\label{lem-4.1}
Let $R^0_\ast>R^{\#}$, where $R^{\#}\approx2.412305$. Then $\mu_\ast(R^0_\ast,\alpha)$ is monotone decreasing with respect to $\alpha$ for $\alpha>0$.
\end{lemma}

\begin{proof}
By \eqref{h-46} and \eqref{a-72},
\begin{align*}
\mu_\ast(R^0_\ast,\alpha)=\frac6{\bar{\sigma}(R^0_\ast)^4 P_0(R^0_\ast)}\left(1+\frac{R^0_\ast P_0(R^0_\ast)}{\alpha}\right)\frac{\frac2{R^0_\ast}+R^0_\ast P_2(R^0_\ast)+\alpha}{2P_1(R^0_\ast)-P_2(R^0_\ast)+\alpha R^0_\ast(P_1(R^0_\ast)-P_2(R^0_\ast))}.
\end{align*}
By differentiation, we compute
\begin{equation*}
\frac{\partial\mu_\ast}{\partial\alpha}(R^0_\ast,\alpha)=\frac6{\bar{\sigma}(R^0_\ast)^4 P_0(R^0_\ast)}
\frac{E(R^0_\ast,\alpha)}{\alpha^2\left[2P_1(R^0_\ast)-P_2(R^0_\ast)+\alpha R^0_\ast(P_1(R^0_\ast)-P_2(R^0_\ast))\right]^2}
\end{equation*}
with
$$
E(x,\alpha)=E_0(x)+E_1(x)\alpha+E_2(x)\alpha^2\quad{\rm for}~x>0~{\rm and}~\alpha>0,
$$
\begin{align*}
E_0(x)&=-xP_0(x)\left(xP_2(x)+\frac2x\right)\left(2P_1(x)-P_2(x)\right),\quad x>0,
\\
E_1(x)&=-2xP_0(x)\left(2+x^2P_2(x)\right)\left(P_1(x)-P_2(x)\right),\quad x>0,
\\
E_2(x)&=\left(P_0(x)+P_2(x)\right)\left[\frac{P_2(x)}{P_0(x)+P_2(x)}-x^2\left(P_1(x)-P_2(x)\right)\right],\quad x>0.
\end{align*}
It is obvious from \eqref{eq2.7} that $E_0(x)<0$ and $E_1(x)<0$ for all $x>0$. To see the sign of the function $E_2(x)$, \eqref{eq2.7} implies
that $P_2(x)/(P_0(x)+P_2(x))<1/2$ for any $x>0$. Using Lemma \ref{lem-2.1} together with the fact $G_2(4)\approx0.553598$ by Matlab, we then know that $E_2(x)<0$ for $x\ge4$. Next, the computation based on Matlab again shows that $E_2>0$ in $(0,x_0)$, $E_2(x_0)=0$ and $E_2<0$ in $(x_0,4)$, where $x_0\approx2.412305$. Hence,
$$
\frac{\partial\mu_\ast}{\partial\alpha}(R^0_\ast,\alpha)<0\quad{\rm for}~R^0_\ast>R^{\#}~{\rm and}~\alpha>0,
$$
as desired.
\end{proof}

\begin{remark}
In comparison with the threshold value $\mu_\ast$ obtained in the work \cite{ZH-1}, which is formally the case $\alpha=\infty$, Lemma \ref{lem-4.1} tells us that when the stationary tumors are of the same size and large enough, our threshold value $\mu_\ast$ is larger and thus the corresponding stationary tumor is more linearly stable.
\end{remark}

The following lemma  gives the impact of angiogenesis on the size of the stationary tumor.

\begin{lemma}
\label{rem-4.2}
$R^0_\ast$ is monotone increasing in $\alpha$, and
$$
\lim_{\alpha\to0}R^0_\ast=0,\quad\lim_{\alpha\to\infty}R^0_\ast=R^0_{\ast,D},
$$
where the positive constant $R^0_{\ast,D}$ solves
$$
P_0(R^0_{\ast,D})=\frac{\tilde{\sigma}}{2\bar{\sigma}}.
$$
\end{lemma}

In fact, $R^0_{\ast,D}$ is exactly that of the model with the Dirichlet boundary condition \eqref{m-13-D}; see (4.52) in \cite{ZH-1}. This lemma is evident from $$
\frac{dR^0_\ast}{d\alpha}=\frac{R^0_\ast P_0^2(R^0_\ast)}{\alpha P_0^2(R^0_\ast)-\alpha^2P'_0(R^0_\ast)}>0,
$$
and \eqref{eq2.17}, \eqref{eq2.18}.

\subsection{Sign of $R^1_\ast$}
In this subsection, we would like to know how the time delay $\tau$  affects the size of the stationary tumor. Recalling that $R_\ast=R^0_\ast+\tau R^1_\ast+O(\tau^2)$ , we are thus now interested in the sign of $R^1_\ast$, for which the equation has been derived in Subsection $4.1$; see \eqref{a-81}.

\begin{proposition}
\label{mainth-a1}
$R^1_\ast>0$ and $R^1_\ast$ is monotone increasing in $\mu$.
\end{proposition}

\begin{proof}
In view of \eqref{a-81}, we first compute from \eqref{eq4.14} that
\begin{equation}
\label{a-80}
\sigma^1_\ast(r)=-\frac{\lambda R^1_\ast}{\alpha+R^0_\ast P_0(R^0_\ast)}\frac{I_0(r)}{I_0(R^0_\ast)}.
\end{equation}
Then, \eqref{eq2.19} and \eqref{eq4.15} imply
\begin{align}
\int^r_0\sigma^1_\ast(l)ldl=&-\frac{\alpha\bar{\sigma}}{\alpha+R^0_\ast P_0(R^0_\ast)}\frac{1-P_0(R^0_\ast)+\alpha R^0_\ast P_0(R^0_\ast)}{\alpha+R^0_\ast P_0(R^0_\ast)}
\frac{rI_1(r)}{I_0(R^0_\ast)}R^1_\ast\nonumber
\\
=&-\frac{\alpha\bar{\sigma}}{\alpha+R^0_\ast P_0(R^0_\ast)}\frac{1-P_0(R^0_\ast)+\alpha R^0_\ast P_0(R^0_\ast)}{\alpha+R^0_\ast P_0(R^0_\ast)}
\frac{r^2P_0(r)I_0(r)}{I_0(R^0_\ast)}R^1_\ast.\label{eq4.16}
\end{align}
In addition, a direct calculation based on \eqref{eq2.20}, \eqref{eq2.21}, \eqref{eq4.23}, \eqref{a-52} and \eqref{ab-1} gives
\begin{align}
&\int^r_0\frac{\partial\sigma^0_\ast}{\partial l}(l)\frac{\partial p^0_\ast}{\partial l}(l)ldl\nonumber
\\
=&
\frac{\mu}{I_0(R^0_\ast)}\left(\frac{\alpha\bar{\sigma}}{\alpha+R^0_\ast P_0(R^0_\ast)}\right)^2
\left(P_0(R^0_\ast)[r^2I_0(r)-2rI_1(r)]
-\frac{r^2[I_1^2(r)-I_0^2(r)]+2rI_0(r)I_1(r)}{2I_0(R^0_\ast)}\right)\nonumber
\\
=&
\mu\left(\frac{\alpha\bar{\sigma}}{\alpha+R^0_\ast P_0(R^0_\ast)}\right)^2
\frac{r^2I_0(r)}{2I_0(R^0_\ast)}
\left(2P_0(R^0_\ast)[1-2P_0(r)]-\frac{I_0(r)[r^2P^2_0(r)-1+2P_0(r)]}{I_0(R^0_\ast)}\right).\label{eq4.21}
\end{align}
On the other hand, we use \eqref{a-51} and \eqref{a-52} to compute
\begin{align}
[\sigma^0_\ast(R^0_\ast)-\tilde\sigma]R^0_\ast R^1_\ast=&
\frac{\alpha\bar{\sigma}}{\alpha+R^0_\ast P_0(R^0_\ast)}[1-2P_0(R^0_\ast)]R^0_\ast R^1_\ast\nonumber
\\
=&\frac{\alpha\bar{\sigma}}{\alpha+R^0_\ast P_0(R^0_\ast)}(R^0_\ast)^3P_0(R^0_\ast)P_1(R^0_\ast)R^1_\ast,\label{eq4.30}
\end{align}
where the last equality is obtained from the fact that
\begin{equation}
\label{eq4.26}
r^2P_0(r)P_1(r)+2P_0(r)=1,
\end{equation}
ensured by \eqref{eq2.13}.
Based on \eqref{eq4.16}--\eqref{eq4.30},
\eqref{a-81} becomes
\begin{equation}
\label{eq4.22}
\frac{\mu}2\alpha\bar{\sigma}
\bigg(1-[4+(R^0_\ast)^2]P^2_0(R^0_\ast)\bigg)
=\bigg(P_0(R^0_\ast)+\alpha R^0_\ast[P_0(R^0_\ast)-P_1(R^0_\ast)]\bigg)P_0(R^0_\ast)R^1_\ast,
\end{equation}
where we have employed \eqref{eq4.26} again.

It immediately follows from \eqref{eq2.7} that the coefficient of $R^1_\ast$ on the right-hand side of \eqref{eq4.22} is positive; $R^1_\ast$ is thus explicitly solved. To further prove $R^1_\ast>0$, by \eqref{eq4.22}, it suffices to show
\begin{equation}
\label{claim-1}
P^2_0(r)<\frac1{4+r^2}\quad{\rm for~all}~r>0.
\end{equation}
In fact, one obtains from \eqref{eq2.2} and \eqref{eq2.5-1} that
$$
I_1^2(r)<I_0(r)I_2(r)+\frac2r I_1(r)I_2(r)=\left[I_0(r)+\frac2r I_1(r)\right]I_2(r),
$$
and
$$
I_2(r)=I_0(r)-\frac2r I_1(r),
$$
which, combined with the definition of $P_0(r)$, implies \eqref{claim-1}.
Finally, it is obvious from \eqref{eq4.22} that $R^1_\ast$ is monotone increasing in $\mu$. The proof is complete.
\end{proof}

\begin{remark}
\label{mainth-r7}
Compared with models without time delays, Proposition \ref{mainth-a1} indicates that the size of stationary solution with time delay is larger, and the larger the tumor aggressiveness parameter $\mu$ is, the greater impact time delays have on the size of the stationary
tumor. It is reasonable because there is more time for the tumor to grow in models with time delays.
\end{remark}

\subsection{First-order terms in $\tau$}

Recalling that $\sigma^1_\ast(r)$ and $R^1_\ast$ have been obtained in Subsection 4.3, we now continue solving the system \eqref{eq4.14}--\eqref{a-85}.
By the first equation in \eqref{a-82}, we compute
\begin{equation}
\label{eq4.25}
\frac{\partial p^1_\ast}{\partial r}(r)=-\frac{\mu}r\int^r_0\left(\frac{\partial\sigma^0_\ast}{\partial l}(l)\frac{\partial p^0_\ast}{\partial l}(l)+\sigma^1_\ast(l)\right)ldl,
\end{equation}
and then $p^1_\ast(r)$ follows from \eqref{eq4.25} and the boundary condition in \eqref{a-82}.
\eqref{eq2.1} together with \eqref{a-83} implies
\begin{equation}
\omega^1_n(r,t)=C_2(t)I_n(r).
\label{a-88}
\end{equation}
According to \eqref{a-83-1}, we need to compute
\begin{align}
\frac{\partial\omega^1_n}{\partial r}(r, t)=&
C_2(t)I_n(r)\left[rP_n(r)+\frac n{r}\right],
\label{ab-2}
\\
\frac{\partial^2\omega^0_n}{\partial r^2}(r,t)
=&\frac{-\lambda}{\alpha+h_n(R^0_\ast)}\frac{I_n(r)}{I_n(R^0_\ast)}
\left[1+\frac{n^2-n}{r^2}-P_n(r)\right]\rho^0_n(t),  \label{ab-3}
\\
\frac{\partial^3\sigma^0_\ast}{\partial r^3}(r)
=&\frac{\alpha\bar{\sigma}}{\alpha+R^0_\ast P_0(R^0_\ast)}\frac{I_0(r)}{I_0(R^0_\ast)}
\left[-\frac1r+\frac2rP_0(r)+rP_0(r)\right], \label{ab-5}
\\
\frac{\partial\sigma^1_\ast}{\partial r}(r)
=&-\frac{\lambda R^1_\ast}{\alpha+R^0_\ast P_0(R^0_\ast)}\frac{I_1(r)}{I_0(R^0_\ast)},\label{eq4.27}
\\
\frac{\partial^2\sigma^1_\ast}{\partial r^2}(r)
=&-\frac{\lambda R^1_\ast}{\alpha+R^0_\ast P_0(R^0_\ast)}\frac{I_0(r)}{I_0(R^0_\ast)}[1-P_0(r)]. \label{ab-6}
\end{align}
Using \eqref{a-83-1}, \eqref{eq4.24}, \eqref{eq4.15}, \eqref{eq4.18}, \eqref{eq4.26} and \eqref{a-88}--\eqref{ab-6}, we arrive at
\begin{equation}
\label{a-89}
\omega^1_n(r, t)=
\frac{I_n(r)}{[\alpha+h_n(R^0_\ast)]I_n(R^0_\ast)}
\left(-\lambda\rho^1_n(t)
       +\frac{\alpha\bar{\sigma}}{\alpha+R^0_\ast P_0(R^0_\ast)}H(n,\alpha,R^0_\ast)R^1_\ast\rho^0_n(t)
\right),
\end{equation}
where
\begin{align}
H(n,\alpha,R^0_\ast)=
&\frac{1-P_0(R^0_\ast)+\alpha R^0_\ast P_0(R^0_\ast)}{\alpha+h_n(R^0_\ast)}
                      \bigg(1-P_n(R^0_\ast)+\frac{n^2-n}{(R^0_\ast)^2}+\alpha h_n(R^0_\ast)\bigg)\nonumber
\\
&+\bigg(\frac{1-P_0(R^0_\ast)-(R^0_\ast)^2(P_0(R^0_\ast)-P_1(R^0_\ast))
                         -\alpha R^0_\ast(P_0(R^0_\ast)-P_1(R^0_\ast))}
                         {\alpha+R^0_\ast P_0(R^0_\ast)}\nonumber
\\
&~~~~~~-\frac{\alpha^2(1-(R^0_\ast)^2(P_0(R^0_\ast)-P_1(R^0_\ast)))}
                      {\alpha+R^0_\ast P_0(R^0_\ast)}\bigg)P_0(R^0_\ast).\label{eq4.33}
\end{align}

In the sequel, as in \cite{ZH-1}, we shall distinguish the two cases: $n\neq1$ and $n=1$, respectively. Let us first consider the case $n\neq1$. To compute $q^1_n(r,t)$, we set
\begin{equation}
\label{eq4.28}
q^1_n=-\mu\omega^1_n+u^{(1)}_n+u^{(2)}_n+u^{(3)}_n+u^{(4)}_n,
\end{equation}
where $u^{(1)}_n$, $u^{(2)}_n$, $u^{(3)}_n$ and $u^{(4)}_n$ satisfy the following equations, respectively:
\begin{equation}
\label{a-92}
    \begin{cases}
      L_nu^{(1)}_n=\mu\frac{\partial\sigma^0_\ast}{\partial r}\frac{\partial q^0_n}{\partial r}, \quad 0<r<R^0_\ast,  \\
      u^{(1)}_n(R^0_\ast,t)=0;
    \end{cases}
\end{equation}
\begin{equation}
\label{a-93}
    \begin{cases}
      L_nu^{(2)}_n=\mu\frac{\partial\omega^0_n}{\partial r}\frac{\partial p^0_\ast}{\partial r}, \quad 0<r<R^0_\ast, \\
      u^{(2)}_n(R^0_\ast, t)=0;
    \end{cases}
\end{equation}
\begin{equation}
\label{a-94}
    \begin{cases}
      L_nu^{(3)}_n=-\mu\frac{\partial\omega^0_n}{\partial t}, \quad 0<r<R^0_\ast, \\
      u^{(3)}_n(R^0_\ast, t)=0;
    \end{cases}
\end{equation}
\begin{equation}
\label{a-95}
    \begin{cases}
      L_nu^{(4)}_n=0, \quad 0<r<R^0_\ast, \\
      u^{(4)}_n(R^0_\ast, t)=\mu\omega^1_n(R^0_\ast,t)-\frac{\partial q^0_n}{\partial r}(R^0_\ast, t)R^1_\ast+\frac{n^2-1}{(R^0_\ast)^2}\rho^1_n(t)-\frac{2(n^2-1)}{(R^0_\ast)^3}R^1_\ast\rho^0_n(t).
    \end{cases}
\end{equation}
By \eqref{eq4.23}, \eqref{eq4.17}, \eqref{a-63} and \eqref{eq4.18}, we further find
\begin{align}
L_nu^{(1)}_n=&
\mu\frac{\alpha\bar{\sigma}}{\alpha+R^0_\ast P_0(R^0_\ast)}
\frac{I_0(r)P_0(r)}{I_0(R^0_\ast)}
\bigg\{
     \frac{\mu\lambda}{\alpha+h_n(R^0_\ast)}\frac{I_n(r)}{I_n(R^0_\ast)}(n+r^2P_n(r))\nonumber
\\
&~~~~~~~~~~~~~~~~~~~~~~~~~~~~~~~~
     +\bigg[\frac{n^2-1}{(R^0_\ast)^2}-\frac{\mu\lambda}{\alpha+h_n(R^0_\ast)}\bigg]
       n\left(\frac r{R^0_\ast}\right)^n
\bigg\}
\rho^0_n(t),\label{a-96}
\end{align}
and the absolute value of the right-hand side of \eqref{a-96} is not larger than $\Phi_1(n)|\rho^0_n(t)|$ for all $0\leq r\le R^0_\ast$, where $\Phi_1(n)$ is a polynomial function of $n$.
Similar estimates can be derived for $u^{(2)}_n$ and $u^{(3)}_n$. Thus, \citep[Lemma 4.6]{ZH-1} guarantees the existence and uniqueness of $u^{(k)}_n$ in $H^2(B_{R^0_\ast})$, $k=1$, $2$, $3$. In addition, solving \eqref{a-95} gives
\begin{equation}
\label{a-109}
u^{(4)}_n(r, t)=C_3(t)r^n
\end{equation}
with
\begin{equation*}
\begin{split}
C_3(t)=&\frac{R^1_\ast\rho^0_n(t)}{(R^0_\ast)^n}
\left\{\frac{\mu}{\alpha+h_n(R^0_\ast)}
\left[\frac{\alpha\bar{\sigma}}{\alpha+R^0_\ast P_0(R^0_\ast)}H(n,\alpha,R^0_\ast)
-\lambda R^0_\ast P_n(R^0_\ast)\right]
-\frac{(n+2)(n^2-1)}{(R^0_\ast)^3}\right\}
\\
&+\frac{1}{(R^0_\ast)^n}\left[\frac{n^2-1}{(R^0_\ast)^2}
-\frac{\lambda\mu}{\alpha+h_n(R^0_\ast)}\right]\rho^1_n(t),
\end{split}
\end{equation*}
where $H(n,\alpha,R^0_\ast)$ is given by \eqref{eq4.33}, hence, $q^1_n$ is established.

To analyze $\rho^1_n(t)$, according to \eqref{a-85}, we first differentiate \eqref{eq4.19} and use \eqref{eq2.4}, \eqref{eq2.14}, \eqref{eq4.26} to get
\begin{equation}
\label{eq4.29}
\frac{\partial^3p^0_\ast}{\partial r^3}(R^0_\ast)=
-\mu\frac{\alpha\bar\sigma}{\alpha+R^0_\ast P_0(R^0_\ast)}
R^0_\ast P_0(R^0_\ast)[1-P_1(R^0_\ast)].
\end{equation}
Next, \eqref{a-52} together with \eqref{ab-1} implies $(p^0_\ast)'(R^0_\ast)=0$, and thus, it follows from \eqref{a-82}, \eqref{a-81} and \eqref{eq4.25} that
$$
\frac{\partial^2p^1_\ast}{\partial r^2}(R^0_\ast)
=-\frac{\mu}{(R^0_\ast)^2}[\sigma^0_\ast(R^0_\ast)-\tilde\sigma]R^0_\ast R^1_\ast-\mu\sigma^1_\ast(R^0_\ast).
$$
Applying \eqref{eq4.15}, \eqref{a-80}, \eqref{eq4.30} and \eqref{eq4.26}, one obtains
\begin{equation}
\label{eq4.31}
\frac{\partial^2p^1_\ast}{\partial r^2}(R^0_\ast)
=\mu\frac{\alpha\bar\sigma}{\alpha+R^0_\ast P_0(R^0_\ast)}
\frac{2P_0(R^0_\ast)+(R^0_\ast)^2P_1(R^0_\ast)+\alpha R^0_\ast[1-P_1(R^0_\ast)]}{\alpha+R^0_\ast P_0(R^0_\ast)}P_0(R^0_\ast)R^1_\ast.
\end{equation}
Using \eqref{a-55}, \eqref{eq4.15}, \eqref{a-58} and \eqref{eq4.20}, we then compute
\begin{equation}
\label{eq4.32}
\frac{\partial^2q^0_n}{\partial r^2}(R^0_\ast,t)
=\left\{\frac{(n^2-1)(n^2-n)}{(R^0_\ast)^4}
+\frac{\mu\lambda}{\alpha+h_n(R^0_\ast)}(1-P_n(R^0_\ast))\right\}\rho^0_n(t).
\end{equation}
Finally, we derive from \eqref{a-89}, \eqref{eq4.28} and \eqref{a-109} that
\begin{align}
\frac{\partial q^1_n}{\partial r}(R^0_\ast, t)=&
\frac{\partial u^{(1)}_n}{\partial r}(R^0_\ast, t)+\frac{\partial u^{(2)}_n}{\partial r}(R^0_\ast, t)+\frac{\partial u^{(3)}_n}{\partial r}(R^0_\ast, t)\nonumber
\\
&+\left(\frac{n(n^2-1)}{(R^0_\ast)^3}
       +\lambda\mu\frac{R^0_\ast P_n(R^0_\ast)}{\alpha+h_n(R^0_\ast)}
\right)\rho^1_n(t)\nonumber
\\
&-\left(\frac{\mu R^0_\ast P_n(R^0_\ast)}{\alpha+h_n(R^0_\ast)}
         \left[\frac{\alpha\bar{\sigma}}{\alpha+R^0_\ast P_0(R^0_\ast)}H(n,\alpha,R^0_\ast)
               +\frac{\lambda n}{R^0_\ast}\right]
         +\frac{n(n+2)(n^2-1)}{(R^0_\ast)^4}
\right)R^1_\ast\rho^0_n(t).\label{a-110}
\end{align}
Substituting \eqref{a-60}, \eqref{eq4.29}--\eqref{a-110} into \eqref{a-85}, we get the equation for $\rho^1_n(t)$:
\begin{align}
\frac{d\rho^1_n(t)}{dt}=
&-\left[\frac{n(n^2-1)}{(R^0_\ast)^3}
       +\lambda\mu\frac{R^0_\ast P_n(R^0_\ast)}{\alpha+h_n(R^0_\ast)}
       -\mu\frac{\alpha\bar{\sigma}}{\alpha+R^0_\ast P_0(R^0_\ast)}
         (R^0_\ast)^2P_0(R^0_\ast)P_1(R^0_\ast)\right]\rho^1_n(t)\nonumber
\\
&-\frac{\partial u^{(1)}_n(R^0_\ast, t)}{\partial r}-\frac{\partial u^{(2)}_n(R^0_\ast, t)}{\partial r}-\frac{\partial u^{(3)}_n(R^0_\ast, t)}{\partial r}+\tilde{H}(\mu,\alpha,n,R^0_\ast) R^1_\ast\rho^0_n(t),\label{eq4.35}
\end{align}
where $\tilde{H}$ is a known function of $\mu$, $\alpha$, $n$, $R^0_\ast$, and satisfies
\begin{equation}
\label{eq4.37}
|\tilde{H}(\mu,\alpha,n,R^0_\ast)|\le\Phi_2(n)
\end{equation}
for some polynomial function $\Phi_2(n)$. Since the explicit expression for $\tilde{H}$ will not be used, we do not write it in detail.

We now proceed to study the asymptotic behavior of $\rho^1_n(t)$ as $t\to\infty$. Applying \citep[Lemma 4.6]{ZH-1} to the problems \eqref{a-92}--\eqref{a-94} and \eqref{a-96} again, we get
\begin{align}
\left|\frac{\partial u^{(1)}_n(R^0_\ast, t)}{\partial r}\right|
+\left|\frac{\partial u^{(2)}_n(R^0_\ast, t)}{\partial r}\right|
+\left|\frac{\partial u^{(3)}_n(R^0_\ast, t)}{\partial r}\right|
\le&\Phi_3(n)\left[|\rho^0_n(t)|+\left|\frac{d\rho^0_n(t)}{dt}\right|\right]\nonumber
\\
\le&\Phi_4(n)|\rho^0_n(t)|\label{eq4.36}
\end{align}
for all $t>0$, where $\Phi_3(n)$, $\Phi_4(n)$ are polynomial functions of $n$, and the last inequality is obtained from \eqref{eq4.10}--\eqref{eq4.2}. Thus, it follows from \eqref{eq4.22} and \eqref{eq4.35}--\eqref{eq4.36} that
\begin{align}
&\left|\frac{d\rho^1_n(t)}{dt}+
\left[\frac{n(n^2-1)}{(R^0_\ast)^3}
       +\lambda\mu\frac{R^0_\ast P_n(R^0_\ast)}{\alpha+h_n(R^0_\ast)}
       -\mu\frac{\alpha\bar{\sigma}}{\alpha+R^0_\ast P_0(R^0_\ast)}
         (R^0_\ast)^2P_0(R^0_\ast)P_1(R^0_\ast)\right]\rho^1_n(t)\right|\nonumber
\\
\le&\Phi_5(n)|\rho^0_n(t)|.\label{eq4.40}
\end{align}
Recall that $\rho^0_n(t)$ exhibits different behaviors under $n\geq2$ and $n=0$. If $n\geq2$,
then by \eqref{eq4.15}, \eqref{h-46} and \eqref{eq4.26}, we can rewrite \eqref{eq4.40} as
\begin{align*}
\left|\frac{d\rho^1_n(t)}{dt}+
\left(1-\frac{\mu}{\mu^0_n}\right)\frac{n(n^2-1)}{(R^0_\ast)^3}
\rho^1_n(t)\right|\le\Phi_5(n)|\rho^0_n(t)|.
\end{align*}
Proposition \ref{l5} asserts that when $0<\mu<\mu_\ast$,
\begin{equation*}
%\label{a-111}
\left|\frac{d\rho^1_n(t)}{dt}+
\left(1-\frac{\mu}{\mu^0_n}\right)\frac{n(n^2-1)}{(R^0_\ast)^3}
\rho^1_n(t)\right|\le Ce^{-(\delta_2/2) n^3t}\quad{\rm for}~t\geq t_0,
\end{equation*}
where $t_0$ is a small positive constant, $C$ is independent of $n$ and $\delta_2$ is defined by \eqref{delta}.
Since \eqref{eq4.38}, \eqref{delta} and the definition of $\mu_\ast$ implies that for $0<\mu<\mu_\ast$,
$$
\left(1-\frac{\mu}{\mu^0_n}\right)\frac{n(n^2-1)}{(R^0_\ast)^3}
\ge\delta_2 n^3>\frac{\delta_2}2 n^3,
$$
using \citep[Lemma 4.7]{ZH-1} we derive $|\rho^1_n(t)|\leq Ce^{-(\delta_2/2)n^3t}$.
While if $n=0$, then \eqref{eq4.40} becomes
\begin{equation*}
\left|\frac{d\rho_0^1(t)}{dt}
-\mu B_0\rho_0^1(t)\right|\leq C|\rho^0_0(t)|.
\end{equation*}
By \eqref{a-66}, we further have for any $\mu>0$ and any $t>0$,
\begin{equation*}
%%\label{a-114}
\left|\frac{d\rho_0^1(t)}{dt}
-\mu B_0\rho_0^1(t)\right|\leq Ce^{-(-\mu B_0)t}\leq Ce^{-(-\mu B_0/2)t}.
\end{equation*}
Applying \citep[Lemma 4.7]{ZH-1} again yields $|\rho_{0}^{1}(t)|\leq Ce^{-(-\mu B_0/2)t}$ for $t>0$.

We next turn to the case $n=1$, where $\rho^0_1(t)\equiv \rho^0_1(0)$. The previous method for establishing $q^1_n(r,t)$ might not work here. Surprisingly, as we shall see, $q^1_1(r,t)$ can be solved explicitly. As a matter of fact, noticing that $\partial_t\omega^0_1$=0,
after a straightforward but lengthy calculation, we get
\begin{equation}
\label{a-119}
L_1\left(q^1_1+\mu\omega^1_1\right)
=\left(\frac{\mu\alpha\bar\sigma}{\alpha+R^0_\ast P_0(R^0_\ast)}\right)^2
\frac{R^0_\ast P_0(R^0_\ast)}{I_0(R^0_\ast)}\rho^0_1(t)
\left(\frac{2I_1(r)}{I_1(R^0_\ast)}\left[I_0(r)-\frac{I_1(r)}{r}\right]-\frac1{R^0_\ast}rI_0(r)\right)
\end{equation}
with the boundary condition
\begin{equation}
\label{a-120}
\left(q^1_1+\mu\omega^1_1\right)(R^0_\ast, t)
=\frac{\mu\alpha\bar\sigma R^0_\ast P_0(R^0_\ast)}{\alpha+R^0_\ast P_0(R^0_\ast)}
\left[-\rho^1_1(t)
+R^1_\ast\rho^0_1(t)\frac{P_0(R^0_\ast)+\alpha R^0_\ast(P_0(R^0_\ast)-P_1(R^0_\ast))}{\alpha+R^0_\ast P_0(R^0_\ast)}\right],
\end{equation}
and
\begin{equation}
\label{eq4.41}
\omega^1_1(r,t)=\frac{\alpha\bar\sigma}{\alpha+R^0_\ast P_0(R^0_\ast)}
\frac{I_1(r)}{I_0(R^0_\ast)}\left[-\rho^1_1(t)
+R^1_\ast\rho^0_1(t)\frac{1-P_0(R^0_\ast)+\alpha R^0_\ast P_0(R^0_\ast)}{\alpha+R^0_\ast P_0(R^0_\ast)}\right].
\end{equation}
Using \eqref{eq2.22} and \eqref{eq2.23}, one derives from \eqref{a-119}
that
\begin{align}
&q^1_1(r, t)\nonumber
\\
=&-\mu\omega^1_1(r,t)+C_4(t)r\nonumber
\\
&
+\left(\frac{\mu\alpha\bar\sigma}{\alpha+R^0_\ast P_0(R^0_\ast)}\right)^2
\frac{R^0_\ast P_0(R^0_\ast)}{I_0(R^0_\ast)}\frac{\rho^0_1(t)}2\bigg(\frac{-I^2_1(r)+I_0(r)I_2(r)}{I_1(R^0_\ast)}
-\frac{1-2I_2(r)}{R^0_\ast}\bigg)r.\label{a-129}
\end{align}
Substituting \eqref{a-129} into \eqref{a-120} and applying \eqref{eq4.26}, \eqref{eq4.22}, we obtain
\begin{equation}
\label{a-130}
C_4(t)=
\frac{\mu\alpha\bar\sigma}{\alpha+R^0_\ast P_0(R^0_\ast)}P_0(R^0_\ast)
\left[-\rho^1_1(t)
+\frac{\mu\alpha\bar\sigma}{\alpha+R^0_\ast P_0(R^0_\ast)}\frac1{2I_0(R^0_\ast)}\rho^0_1(t)\right].
\end{equation}

It remains to solve for $\rho^1_1(t)$. For this, according to \eqref{a-85}, we need use \eqref{eq4.22} and \eqref{eq4.41}--\eqref{a-130} to compute
\begin{align}
\frac{\partial q^1_1}{\partial r}(R^0_\ast, t)=&
\frac{\mu\alpha\bar\sigma}{\alpha+R^0_\ast P_0(R^0_\ast)}
\bigg\{(1-2P_0(R^0_\ast))\rho^1_1(t)\nonumber
\\
&~~~~~~~
-\frac{R^1_\ast\rho^0_1(t)P_0(R^0_\ast)}{\alpha+R^0_\ast P_0(R^0_\ast)}
[(R^0_\ast)^2P_1(R^0_\ast)+2P_0(R^0_\ast)+\alpha R^0_\ast(1-P_1(R^0_\ast))]\bigg\}.\label{a-131}
\end{align}
Then substituting \eqref{a-60}, \eqref{eq4.29}--\eqref{eq4.32}, \eqref{a-131} into \eqref{a-85} and by the fact that
$$
\frac{\lambda}{\alpha+h_1(R^0_\ast)}=\frac{\alpha\bar\sigma}{\alpha+R^0_\ast P_0(R^0_\ast)}R^0_\ast P_0(R^0_\ast),
$$
we immediately arrive at
\begin{align*}
\frac{d\rho^1_1(t)}{dt}=0,
\end{align*}
which implies that $\rho_1^1(t)\equiv\rho_1^1(0)$.

\begin{remark}
\label{mainth-r8}
By analyzing the first-order terms in $\tau$, we conclude that after ignoring $O(\tau^2)$ terms,
introducing the time delay into the system would not affect the stability.
\end{remark}

\section*{Acknowledgments}
This work was partly supported by the National Natural Science Foundation of China (No. 11861038 and No. 11771156).

\end{document}